\documentclass[10pt, article]{amsart}

\usepackage{ae} 
\usepackage[T1]{fontenc}
\usepackage[cp1250]{inputenc}
\usepackage{amsmath}
\usepackage{amssymb, amsfonts,amscd,verbatim}
\usepackage{mathtools}
\usepackage{MnSymbol}

\usepackage[normalem]{ulem}
\usepackage{hyperref}
\usepackage{indentfirst}
\usepackage{latexsym}
\input xy
\xyoption{all}

\usepackage{amsmath}    

\theoremstyle{plain}
\newtheorem{Pocz}{Poczatek}[section]
\newtheorem{Proposition}[Pocz]{Proposition}

\newtheorem{Theorem}[Pocz]{Theorem}
\newtheorem{Corollary}[Pocz]{Corollary}

\newtheorem{Lemma}[Pocz]{Lemma}

\newtheorem{Example}[Pocz]{Example}

\theoremstyle{definition}
\newtheorem{Definition}[Pocz]{Definition}

\theoremstyle{remark}
\newtheorem{Remark}[Pocz]{Remark}

\numberwithin{equation}{section}

\author{Thomas ~ Weighill}
\address{University of Tennessee, Knoxville, USA}
\email{tweighil@vols.utk.edu}

\title{On spaces with connected Higson coronas}

\date{ \today
}


\begin{document}

\maketitle

\begin{abstract}
In this paper, we characterise metric spaces which have topologically connected Higson coronas. The characterisation is given by a natural categorical condition applied in the coarse category. We also give a characterisation in terms of coarse cohomology, and consider the special case of finitely generated groups and the more general case of abstract coarse spaces. Along the way, we exhibit some connections to the notion of $\omega$-excisive decomposition introduced by Higson, Roe and Yu, and give a categorical characterisation of such decompositions.  
\end{abstract}

\tableofcontents

\section{Introduction}
It is a well-known fact that for any metric space $X$ with distance function $d$, the distance function 
$$
\rho(x,y) = \mathsf{min}(d(x,y), 1)
$$
gives the same topology on $X$ as $d$ does. We may thus say that topology involves ``small-scale'' properties of metric spaces, since it ignores large distances. Recently, there has been some interest in considering ``large-scale'', or ``coarse'', properties of metric spaces, motivated by geometric group theory (beginning with Gromov~\cite{Gromov93}) and the study of geometric and topological invariants of manifolds (see for example~\cite{RoeIndex}). Also of note is a result of Yu~\cite{Yu98}, who proved the Novikov Conjecture for finitely generated groups whose classifying space has the homotopy type of a finite CW complex and which, as metric spaces with the word-length metric, have finite asymptotic dimension. Asymptotic dimension is a coarse notion of dimension introduced by Gromov~\cite{Gromov93}.

In addition to studying coarse structures directly, two important ways to study the coarse properties of metric spaces are as follows:
\begin{enumerate}
\item studying metric spaces as objects in the appropriate category, for example the coarse category (defined in the next section), and
\item associating to each metric space a topological space or algebra which captures coarse properties of the space, for example the Higson corona~\cite{Higson90} or the uniform Roe algebra~\cite{RoeAnIndex, RoeCoarse} (see also~\cite{RoeIndex, RoeLectures}).
\end{enumerate}
In this paper, we will exhibit a connection between a condition stated in the language of the coarse category and a topological condition on the Higson corona. For a proper metric space $X$, the \emph{Higson corona} of $X$, denoted by $\nu X$, is a compact topological space which captures coarse properties of $X$. It was introduced in~\cite{Higson90}, motivated by considerations in index theory, and is defined for proper metric spaces as the complement of $X$ in the so-called \emph{Higson compactification} of $X$. While the Higson compactification is only defined for proper metric spaces (or for proper coarse spaces in the sense of~\cite{RoeLectures}), the Higson corona can be defined for arbitrary metric spaces (and arbitrary coarse spaces). We recall this definition here, following~\cite{RoeLectures}. 

\begin{Definition}
Let $X$ be a metric space. Given a bounded map $f: X \rightarrow \mathbb{C}$ (not necessarily continuous) to the complex numbers, $f$ is said to be \emph{slowly oscillating} if for every $\varepsilon > 0$ and $R > 0$, there is a bounded set $B$ such that
$$
d(x, x') \leq R \Rightarrow d(f(x), f(x')) \leq \varepsilon
$$
for $x, x' \in X \setminus B$. 
\end{Definition}

We say that a bounded map $f: X \rightarrow \mathbb{C}$ \emph{tends to zero at infinity} if for every $\varepsilon > 0$, there is a bounded set $B$ in $X$ such that $|f(x)| \leq \varepsilon$ whenever $x \notin B$. Let $B_h(X)$ be the set of bounded slowly oscillating functions from $X$ to $\mathbb{C}$, and let $B_0(X)$ be the set of bounded functions from $X$ to $\mathbb{C}$ which tend to zero at infinity. It is easy to check that $B_h(X)$ is a unital ${C}^\ast$-algebra with the sup-norm and pointwise operations, and that $B_0(X)$ is a closed ideal of $B_h(X)$. 

\begin{Definition}
The \emph{Higson corona} of $X$, denoted by $\nu X$, is the spectrum of the $C^\ast$-algebra $B_h(X)/B_0(X)$.
\end{Definition}
In other words, $\nu X$ is the unique (up to homeomorphism) compact Hausdorff space whose algebra $C(\nu X)$ of continuous complex-valued functions is $\ast$-isomorphic to $B_h(X)/B_0(X)$ (for more on the theory of $C^\ast$-algebras, see for example~\cite{Arveson}). For proper metric spaces this definition coincides with the original definition in terms of the Higson compactification (see Lemma~2.40 in~\cite{RoeLectures}). The study of coarse properties of metric spaces can be viewed as the study of properties which are invariant under \emph{coarse equivalence} of metric spaces (defined in the next section). It turns out that coarsely equivalent metric spaces have homeomorphic Higson coronas~\cite{RoeLectures}, so that the Higson corona is indeed a coarse invariant of a space. A good example of the relationship between coarse properties of a space and topological properties of its Higson corona is Theorem 7.2 of~\cite{Dranishnikov00}, where it was shown that if the asymptotic dimension of a proper metric space is finite, then the asymptotic dimension coincides with the covering dimension of its Higson corona. 

In this paper, we consider a very simple topological condition on the Higson corona, namely that of connectedness. This condition can be stated very easily in terms of the algebra $B_h(X)/B_0(X)$. Note that any topological space $X$ is disconnected if and only if it admits a continuous non-constant map to the discrete space $\{0,1\} \subset \mathbb{C}$. As an immediate consequence, we obtain the following result:

\begin{Lemma} \label{algcon}
Let $X$ be a compact Hausdorff space and $C(X)$ its $C^\ast$-algebra of continuous complex-valued functions. Then $X$ is disconnected if and only if $C(X)$ contains a non-trivial idempotent element. In particular, the Higson corona of a space $X$ is connected if and only if $B_h(X)/B_0(X)$ contains no non-trivial idempotent elements.
\end{Lemma}

The main result of this paper will be to show that connectedness of the Higson corona can be characterised by a categorical condition (condition $\mathsf{(C)}$ in Section~\ref{SecHigson}) which is a natural generalization of the notion of connectedness in topological spaces, stated in the language of the coarse category. This result motivates the study of categorical conditions in the coarse category. Note that, in general, any condition stated in categorical language in the coarse category is automatically invariant under coarse equivalence (since coarse equivalences are isomorphisms in this category). We will also give one further connection between a categorical notion and a well-known coarse condition in this paper, namely, we give a categorical characterisation of the notion of $\omega$-excisive decomposition introduced in~\cite{HigsonMV}. It would be interesting in the future to investigate what other categorical conditions stated in terms of the coarse category turn out to correspond to well-known or interesting coarse properties of spaces.


\section{Preliminaries} \label{SecPrelim}
Throughout this paper, we will often deal with the disjoint union $X + Y$ of two sets $X$ and $Y$. For convenience, we will not distinguish between (relations on) the set $X$ and (relations on) the image of $X$ under the inclusion $\iota_X: X \rightarrow X + Y$. For a subset $A$ of a metric space $X$ and $R > 0$, we denote the set $\{x \in X \mid d(x, A) < R\}$ by $B(A, R)$. Throughout the paper we will make use of elementary category theoretic notions, most importantly the notion of coproduct and pushout, although we give explicit descriptions of the universal property in question whenever possible. For an introduction to category theory, we direct the reader to~\cite{MacLane}.

Let $X$ and $Y$ be metric spaces, and let $f: X \rightarrow Y$ be a map. We call the map $f$ $\rho$-\emph{bornologous}, where $\rho$ is a function $\rho: [0, \infty) \rightarrow [0, \infty)$, if for any points $x, x'$ in $X$, 
$$
d_X(x,x') \leq R \Rightarrow d_Y(f(x), f(x')) \leq \rho(R).
$$
The map $f$ is called \emph{bornologous} if it is $\rho$-bornologous for some $\rho$. In other words, $f$ is bornologous if and only if for every $R > 0$ there is an $S > 0$ such that $d_X(x,x') \leq R \Rightarrow d_Y(f(x), f(x')) \leq S$. It is easy to check that the composite of two bornologous maps is again bornologous. Consequently, metric spaces together with bornologous maps form a category, which we denote throughout by $\mathbf{Met}_\mathbf{Born}$. A map $f$ between metric spaces is called \emph{proper} if the inverse image of any bounded set under $f$ is bounded. A map is called \emph{coarse} if it is both bornologous and proper. Proper maps are closed under composition, so coarse maps form a subcategory of $\mathbf{Met}_\mathbf{Born}$.

Two maps $f$ and $g$ from a metric space $X$ to a metric space $Y$ are \emph{close} if there is a $R > 0$ such that $d_Y(f(x), g(x)) \leq R$ for all $x \in X$. A map $f$ from $X$ to $Y$ is a \emph{coarse equivalence} if there exists a bornologous map $f^\ast: Y \rightarrow X$ such that $f f^\ast$ and $f^\ast f$ are close to the respective identities. Note that the bijective coarse equivalences are precisely the isomorphisms in $\mathbf{Met}_\mathbf{Born}$, and that coarse equivalences are always proper. Coarse equivalences are also coarsely surjective: if $u: X \rightarrow Y$ is a coarse equivalence, then there is an $R > 0$ such that $Y \subseteq B(\mathsf{Im}(u), R)$.

It is often convenient to consider the category whose objects are metric spaces and whose morphisms are equivalence classes of coarse maps under the closeness relation (note that in in this category, coarse equivalences represent isomorphisms). In this paper, we will call this category the \emph{coarse category of metric spaces}, by analogy with the coarse category as defined in~\cite{RoeIndex} (where objects are abstract coarse spaces). Note that composition is well-defined in this category, since if $f$ is close to $g$ and $h$ is close to $k$, where $f$,$g$, $h$ and $k$ are bornologous maps, then $hf$ is close to $kg$ whenever these composites are defined. 

\section{Coarse coproducts}\label{SecCop}
In this section we introduce coarse coproducts of metric spaces and prove some basic results about them.

\begin{Definition}
Let $X$ and $Y$ be two metric spaces, and let $x_0 \in X$ and $y_0 \in Y$ be two arbitrary points (which we will call the \emph{base points} for the coproduct). Then the \emph{coarse coproduct} of $(X, d_X)$ and $(Y, d_Y)$ is the space $(X + Y, d_{X+Y})$ whose underlying set is the disjoint union of $X$ and $Y$ and where the distance $d_{X+Y}$ is defined as follows:
\[ d_{X+Y}(a,b) = \begin{cases} 
      d_X(a, b) & a,b \in X \\
      d_Y(a,b)  & a,b \in Y \\
      d_X(a, x_0) + 1 + d_Y(y_0, b) & a \in X, b \in Y .
   \end{cases}
\]
\end{Definition}

\begin{Proposition}\label{coprod1}
Let $X$ and $Y$ be metric spaces with coarse coproduct $X + Y$ and let $\iota_X: X \rightarrow X + Y$ and $\iota_Y: Y \rightarrow X+Y$ be the evident isometric embeddings. Then
\begin{itemize}
\item[(1)] $X + Y$ (together with $\iota_X$ and $\iota_Y$) is the coproduct in $\mathbf{Met}_\mathbf{Born}$ of $X$ and $Y$, i.e.~if there are bornologous maps $f: X \rightarrow Z$ and $g: Y \rightarrow Z$, then there exists a unique bornologous map $h: X + Y \rightarrow Z$ such that $h \iota_X = f$ and $h \iota_Y = g$;
\item[(2)] in the notation of (1) above, both $f$ and $g$ are proper if and only if $h$ is proper;
\item[(3)] in the notation of (1) above, if $h': X + Y \rightarrow Z$ is a map such that $h' \iota_X$ is close to $f$ and $h' \iota_Y$ is close to $g$, then $h'$ is close to $h$.
\end{itemize}
\end{Proposition}
\begin{proof}
For (1), define $h$ to coincide with $f$ on $X$ and $g$ on $Y$. It remains to show that $h$ is bornologous, since then $h$ is clearly unique with the desired property. Let $r = d_Z(f(x_0), g(y_0))$, where $x_0 \in X$ and $y_0 \in Y$ are the chosen base points, and suppose $f$ and $g$ are $\rho$- and $\sigma$-bornologous respectively. Since $h$ is clearly bornologous on $X$ and $Y$, it remains to consider points $a \in X$ and $b \in Y$ with $d(a,b) \leq R$. We have
$$
d(h(a), h(b)) = d(f(a), g(b)) \leq d(f(a), f(x_0)) + r + d(g(y_0), g(b)) 
$$
Since $d(a, x_0) \leq d(a, b) \leq R$ and $d(y_0, b) \leq d(a, b) \leq R$ in $X+Y$, we have
$$
d(h(a), h(b)) \leq \rho(R) + r + \sigma(R)
$$
which gives the required result. For (2), note that a subset of $X+Y$ is bounded if and only if its restrictions to both $X$ and $Y$ are bounded. (3) is easy to check.
\end{proof}

It follows from the proposition above that the coarse coproduct of $X$ and $Y$ is defined up to bijective coarse equivalence -- that is, if different base points $x_1 \in X$ and $y_1 \in Y$ are chosen for the construction, then the resulting coarse coproduct is coarsely equivalent to the one with base points $x_0$, $y_0$ via the identity set map.

The proposition above also shows that not only is $X+Y$ the coproduct in $\mathbf{Met}_\mathbf{Born}$, but it also gives the coproduct in the subcategory of coarse maps, as well as in the coarse category of metric spaces. As usual, the existence of binary coproducts gives the existence of finite coproducts in all these categories. It is easy to show that arbitrary coproducts (for example, an uncountable coproduct of singleton spaces) do not exist in $\mathbf{Met}_\mathbf{Born}$. Countable coproducts do exist in $\mathbf{Met}_\mathbf{Born}$, however, as the following proposition shows.

\begin{Proposition} \label{countablecop}
Let $X_1, X_2, \ldots$ be a countable family of metric spaces and $x_1 \in X_1, x_2 \in X_2, \ldots$ chosen base points in each space. Let $\sum_i X_i$ be the metric space whose underlying set is the disjoint union of $X_1, X_2, \ldots$ and whose distance $d$ is defined as follows:
\[ d(a,b) = \begin{cases} 
       0 & a = b \\
      d_{X_i}(a, b) + i & a,b \in X_i,\ a \neq b \\
      d_{X_i}(a, x_i) + i + j + d_{X_j}(b, x_j) & a \in X_i, b \in X_j, i\neq j  
   \end{cases}
\]
Then $\sum_i X_i$, together with the obvious injections $(\iota_i)_{i \geq 1}$, is the coproduct of the $X_i$ in $\mathbf{Met}_\mathbf{Born}$.
\end{Proposition}
\begin{proof}
Suppose $Z$ is a metric space and $f_i: X_i \rightarrow Z$ a family of bornologous maps. Let $f: \sum_i X_i \rightarrow Z$ be the induced set map. Then for any $R > 0$, there is a $k \in \mathbb{N}$ such that $d(a, b) \leq R \Rightarrow a = b$ whenever $a \in X_i$ and $b \in X_j$ with $\mathsf{max}(i,j) \geq k$. For $f$ to be bornologous it is thus enough for it to be bornologous on subspaces of the form 
$$
\bigcup_{i \leq k} X_i \subseteq \sum_i X_i
$$
which is easy to show using similar arguments to Proposition~\ref{coprod1}.
\end{proof}

Note that it is not true in general that, in the notation of the proof, the map $f: \sum_i X_i \rightarrow Z$ is proper whenever the $f_i$ are. Moreover, if $f' \iota_i$ is close to $f_i$ for every $i$ for some bornologous map $f'$, then $f'$ need not be close to $f$.

If one changes the definition of the metric $d$ defined in Proposition~\ref{countablecop} to
\[ d(a,b) = \begin{cases} 
      d_{X_i}(a, b) & a,b \in X_i \\
      d_{X_i}(a, x_i) + i + j + d_{X_j}(b, x_j) & a \in X_i, b \in X_j, i\neq j  
   \end{cases}
\]
then one obtains a space $\square_i X_i$ with the following universal property: for any family $f_i: X_i \rightarrow Z$ of $\rho$-bornologous maps, there is a unique bornologous map $f: \square_i X_i \rightarrow Z$ such that $f\iota_i = f_i$ for each $i$. Note that in this case the $f_i$ must share the same function $\rho$. 

This construction is well known in the case when $X = G$ is a finitely generated group with the word length metric and $X_i = G/G_i$ is a sequence of finite quotients of $G$ such that every finite index normal subgroup of $G$ contains some $G_i$. The space $\square_i X_i$ is then (up to bijective coarse equivalence) the \emph{box space} of $G$ introduced in~\cite{RoeLectures}. An important result about box spaces is as follows: a residually finite group $G$ (i.e.~one which admits such a sequence $G_i$) is amenable if and only if the box space satisfies Yu's Property A (see~\cite{NowakYu}).

\section{Connectedness of the Higson corona} \label{SecHigson}
Given any category with coproducts, there are a number of conditions on an object $X$ which in the category of topological spaces and continuous maps all reduce to the usual notion of topological connectedness. Some of these conditions are listed and compared in~\cite{GJanelidze}. In this paper, we will use the following condition from this list:
\begin{itemize}
\item any morphism $f$ from $X$ to a coproduct $Y + Z$ factors through a coproduct injection, i.e.~there exists either a map $g: X \rightarrow Y$ such that $\iota_Y g = f$ or a map $h: X \rightarrow Z$ such that $\iota_Z h = f$.
\end{itemize}
It is easy to check that this capture the notion of connectedness in the case of topological spaces when applied in the category of topological spaces and continuous maps. In the coarse category of metric spaces, this condition becomes the following condition on a metric space $X$:

\begin{itemize}
\item[$\mathsf{(C)}$] for every coarse map $f: X \rightarrow Y + Z$, there exists either a coarse map $g: X \rightarrow Y$ such that $\iota_Y g$ is close to $f$ or a coarse map $h: X \rightarrow Z$ such that $\iota_Z h$ is close to $f$.
\end{itemize}

\begin{Theorem}\label{disconthm}
For a metric space $X$, the following are equivalent:
\begin{itemize}
\item[(a)] $X$ doesn't satisfy $\mathsf{(C)}$;
\item[(b)] there are two unbounded subsets $A$ and $B$ of $X$ such that
\begin{itemize}
\item $X = A \cup B$, and
\item for any $R > 0$, there is a bounded set $C_R \subseteq X$ such that 
$$a \in (A \setminus C_R) \wedge  b \in (B \setminus C_R) \Rightarrow d_X(a,b) \geq R.$$ 
\end{itemize}
\item[(c)] $X$ is (bijectively) coarsely equivalent to a coarse coproduct $Y + Z$ where neither $Y$ nor $Z$ is bounded;
\item[(d)] there exists a coarse map $f: X \rightarrow \mathbb{Z}$ such that the image of $f$ has no maximum or minimum.
\end{itemize}
\end{Theorem}
\begin{proof}
(a) $\Rightarrow$ (b): Suppose $f: X \rightarrow Y + Z$ is a coarse map such that $f$ does not factor, up to closeness, through either $\iota_Y$ or $\iota_Z$, and suppose that $f$ is $\rho$-bornologous. Then in particular, neither $\mathsf{Im}(f) \cap Y$ nor $\mathsf{Im}(f) \cap Z$ are bounded.  Let $A = f^{-1}(Y)$, $B = f^{-1}(Z)$; since $f$ is bornologous, neither $A$ nor $B$ are bounded subspaces of $X$. For any $R > 0$,  let $K = B(y_0, \rho(R)) \cup B(z_0, \rho(R))$ where $y_0 \in Y$ and $z_0 \in Z$ are the base points of the coproduct. Since $f$ is proper, $C_R = f^{-1}(K)$ is bounded. If $a \in (A \setminus C_R)$ and $b \in (B \setminus C_R)$, then using the definition of $Y+Z$, we have that 
$$d(f(a), f(b)) \geq 2\rho(R) + 1 > \rho(R),$$
so that $d(a,b) \geq R$ as required.

(b) $\Rightarrow$ (c): It is easy to see that $A \cap B$ has to be bounded, so that $B \setminus A$ must be non-empty. Choose points $a_0 \in A$ and $b_0 \in B \setminus A$. Note that we may choose the $C_R$ such that for all $R > 0$, $\{a_0, b_0\} \subseteq C_R$. Let $W = A + (B \setminus A)$, choosing $a_0$ and $b_0$ as base points and where the metric on $A$ and $B \setminus A$ is induced by $X$. The identity set map $W \rightarrow X$ is clearly bornologous, so it remains to prove that it has bornologous inverse. The inverse is clearly bornologous on $A$ and $B \setminus A$, so let $a \in A$ and $b \in B \setminus A$, and choose a bounded subset $C_R$ corresponding to the value $R = d_X(a, b)$. Let $D$ be the diameter of $C_R$. Then one of $a$ and $b$ must be in $C_R$, so
$$
d_W(a, b) = d_X(a, a_0) + d_X(b, b_0) + 1 \leq d_X(a, b) + 1 + 2D
$$
since $\{a_0, b_0\} \subseteq C_R$. This shows that the inverse of $i$ is bornologous, since $D$ depends only on $d_X(a,b)$.

(c) $\Rightarrow$ (d): Suppose $X$ is coarsely equivalent to $Y + Z$ with chosen base points $y_0$, $z_0$. Define a map $g: Y \rightarrow \mathbb{Z}$ by $g(y) = k$, where $k \leq d_Y(y, y_0) < k+1$. Define a map $h: Z \rightarrow \mathbb{Z}$ by $h(z) = -k$, where $k \leq d_Z(z, z_0) < k+1$. Clearly both $g$ and $h$ are coarse, so the induced map $f: Y + Z \rightarrow \mathbb{Z}$ is also coarse. Moreover, the image of $f$ has no maximum or minimum since $Y$ and $Z$ are unbounded. Finally, composing with the coarse equivalence from $X$ to $Y + Z$ gives the required map.

(d) $\Rightarrow$ (a): This follows from the fact that $\mathbb{Z}$ is (bijectively coarse equivalent to) the coarse coproduct of $\mathbb{N}$ with itself.
\end{proof}

\begin{Example}
The metric spaces $\mathbb{Z}$ and $\mathbb{R}$ do not satisfy $\mathsf{(C)}$ (indeed, they are both coarsely equivalent to the coarse coproduct $\mathbb{N} + \mathbb{N}$). The metric space $\{n^2 \mid n\in \mathbb{N}\}$ also does not satisfy $\mathsf{(C)}$ as can be seen from (b) in the above theorem (take $A$ to be the even numbers and $B$ the odd ones). It is easy to show using condition (b) in Theorem~\ref{disconthm}, however, that the metric space $\mathbb{N}$ does satisfy $\mathsf{(C)}$. In particular, there are no surjective coarse maps from $\mathbb{N}$ to $\mathbb{Z}$.
\end{Example}

In topological spaces, a space $X$ is disconnected if and only if it admits a non-trivial map to the two element  discrete space. Thus the space $\mathbb{Z}$ in some sense plays the role of the two element discrete space for condition $\mathsf{(C)}$.

\begin{Corollary}
If $f : X \rightarrow Y$ is a surjective coarse map and $X$ satisfies $\mathsf{(C)}$, then $Y$ satisfies $\mathsf{(C)}$.
\end{Corollary}
\begin{proof}
This follows from condition (d) in Theorem~\ref{disconthm}.
\end{proof}

Recall from~\cite{RoeLectures} that the map which takes a metric space $X$ to its Higson corona $\nu X$ extends to a functor $\nu$ from the coarse category of metric spaces to the category of compact Hausdorff spaces and continuous maps (in fact, the result in~\cite{RoeLectures} is stated only for the case of proper metric/coarse spaces, but the proof works for arbitrary metric/coarse spaces). In particular, coarsely equivalent metric spaces have homeomorphic Higson coronas. It turns out that the functor $\nu$ preserves coproducts, as we will now show.

\begin{Lemma}\label{slowcop}
Let $X$ and $Y$ be metric spaces with coarse coproduct $X + Y$. Let $f: X \rightarrow \mathbb{C}$ and $g: Y \rightarrow \mathbb{C}$ be maps to the complex numbers. Then the map $h: X + Y \rightarrow \mathbb{C}$ which agrees with $f$ on $X$ and with $g$ on $Y$ is slowly oscillating/tends to zero at infinity if and only if both $f$ and $g$ are slowly oscillating/tend to zero at infinity.
\end{Lemma}
\begin{proof}
The equivalence for tending to zero at infinity is clear, as is the fact that if $h$ is slowly oscillating, then $f$ and $g$ both are. Suppose then that $f$ and $g$ are slowly oscillating. Let $\varepsilon > 0$ and $R > 0$. Since $f$ and $g$ are slowly oscillating, there are bounded sets $B_1$ in $X$ and $B_2$ in $Y$ such that, outside of $B_1 \cup B_2 \subseteq X+Y$, $d(a, b) \leq R \Rightarrow d(h(a), h(b)) \leq \varepsilon$ whenever $a$ and $b$ are either both in $X$ or both in $Y$. Let $x_0 \in X$ and $y_0 \in Y$ be the base points chosen for $X + Y$, and let $C = B(x_0, R) \cup B(y_0, R)$. Then for $x, x' \in (X + Y) \setminus (B_1 \cup B_2 \cup C)$, we have that $x$ and $x'$ are either both in $X$ or both in $Y$. Consequently, $d(h(x), h(x')) \leq \varepsilon$ as required.

\end{proof}

\begin{Proposition} \label{Higsoncop}
The Higson corona preserves binary coarse coproducts. That is, if $X$ and $Y$ are metric spaces, then $\nu X + \nu Y$ and $\nu(X+Y)$ are homeomorphic.
\end{Proposition}
\begin{proof}
Consider the algebras $C(\nu(X) + \nu(Y)) \cong C(\nu X) \times C(\nu Y)$ and $C(\nu(X+Y))$. There is a canonical $\ast$-homomorphism $F: C(\nu (X + Y)) \rightarrow C(\nu X)\times C(\nu Y)$ which sends an equivalence class of maps $[f]$ to the pair $([f\iota_X], [f\iota_Y])$. It follows from Lemma~\ref{slowcop} that this map has trivial kernel and is surjective, so $F$ is an isomorphism. Thus we obtain that $\nu X + \nu Y$ and $\nu(X+Y)$ are homeomorphic.
\end{proof}

We are now ready to state the main result of this paper. 

\begin{Theorem}\label{discon1}
The following are equivalent for a metric space $X$:
\begin{itemize}
\item[(a)] $X$ satisfies $\mathsf{(C)}$;
\item[(b)] the Higson corona of $X$ is (topologically) connected;
\item[(c)] $B_h(X)/B_0(X)$ does not contain a non-trivial idempotent element.
\end{itemize}
\end{Theorem}
\begin{proof}
(b) $\Rightarrow$ (a): This follows from Proposition~\ref{Higsoncop} and the fact that for an unbounded metric space, the Higson corona is non-empty.

(b) $\Leftrightarrow$ (c): This follows from Lemma~\ref{algcon}.

(a) $\Rightarrow$ (c): Suppose that (c) doesn't hold. Then there is a slowly oscillating map $f: X \rightarrow \mathbb{C}$ such that $[f^2 - f] = [0]$, with $[f] \neq [0]$, $[f] \neq [1]$. This means that for any $\varepsilon > 0$, there is a bounded set $C$ such that the image of $X \setminus C$ under $f$ is contained in $B(0, \varepsilon) \cup B(1, \varepsilon)$. In particular, one can choose $C$ such that the image of $X \setminus C$ under $f$ is contained in $B(0,1/4) \cup B(1, 1/4)$. Let $A = f^{-1}(B(0, 1/4))$ and $B = X \setminus A$. The non-triviality of $[f]$ ensures that neither $A$ nor $B$ are bounded. It follows that for any $R > 0$, we can choose a bounded set $C'$ such that 
$$d(x, x') \leq R \Rightarrow d(f(x), f(x')) \leq 1/4$$
for any $x, x' \notin C'$, and such that
$$f(X \setminus C') \subseteq B(0, 1/4) \cup B(1, 1/4).$$
In particular, if $d(x,x') \leq R$ and $x, x' \notin C'$, then $x$ and $x'$ must either both be in $A$ or both be in $B$. Thus $A$ and $B$ satisfy the conditions in (b) in Theorem~\ref{disconthm}, so we obtain the required contradiction.
\end{proof}

\section{$\omega$-Excisive decompositions and coarse pushouts}\label{SecExc}
In this section we make some connections between the work done so far and the notion of $\omega$-excisive decomposition found in~\cite{HigsonMV}. We do so via a result (Theorem~\ref{pushout} below) which in its own right further motivates the study of categorical conditions in the coarse category. Recall the following definition from~\cite{HigsonMV}.

\begin{Definition}
Let $X$ be a metric space and let $A$ and $B$ be closed subspaces with $X = A \cup B$. Then $X = A \cup B$ is an \emph{$\omega$-excisive decomposition} if for each $R > 0$ there exists a $S > 0$ such that $B(A, R) \cap B(B, R) \subseteq B(A \cap B, S)$.
\end{Definition}
Such decompositions are important because they give rise to Mayer-Vietoris sequences at the level of coarse cohomology as well as at the level of $K$-theory of uniform Roe algebras~\cite{HigsonMV}. In particular, this allows $\omega$-excisive decompositions to be used to prove the coarse Baum-Connes conjecture for certain spaces~\cite{YuBC}. We now show that such decompositions amount to pushouts in the coarse category.

\begin{Theorem}\label{pushout}
Let $X$ be a metric space and let $A$ and $B$ be closed subspaces with $X = A \cup B$. Then $X = A \cup B$ is an $\omega$-excisive decomposition if and only if $A \cap B$ is non-empty and the diagram of inclusions
\begin{equation}\label{pushout}
\begin{gathered}
\xymatrix{
& X & \\
A \ar[ru] & & B \ar[lu] \\
& A\cap B \ar[ru] \ar[lu]\\
}
\end{gathered}
\end{equation}
is a pushout in the coarse category of metric spaces, i.e.~for any coarse maps $f: A \rightarrow Y$ and $g: B \rightarrow Y$ which are close on $A \cap B$, there is a unique-up-to-closeness map $h: X \rightarrow Y$ such that $h$ is close to $f$ on $A$ and close to $g$ on $B$.
\end{Theorem}
\begin{proof}
$(\Rightarrow)$: Suppose $X = A \cup B$ is a $\omega$-excisive decomposition, and that $f: A \rightarrow C$ and $g: B \rightarrow C$ are two coarse maps such that $f$ and $g$ are close on $A \cap B$, with $f$ $\rho$-bornologous and $g$ $\sigma$-bornologous. Define $h: X \rightarrow C$ to be $f$ on $A$ and $g$ on $X \setminus A$. It remains to show that $h$ is coarse. It is clearly coarse on $A$ and $B$, so it remains to consider $a \in A$, $b \in X \setminus A$. Suppose $d(a, b) \leq R$. Then $a, b \in B(A, R) \cap B(B, R)$, so by hypothesis, there is an $S$ such that $a, b \in B(A \cap B, S)$. Suppose $d(a, c_1) \leq S + 1$ and $d(b, c_2) \leq S + 1$ for $c_1, c_2 \in A \cap B$. The distance $d(h(a), h(b)) = d(f(a), g(b))$ is bounded above by
$$
d(f(a), f(c_1)) + d(f(c_1), f(c_2)) + d(f(c_2), g(c_2)) + d(g(c_2), g(b)).
$$
The first and last terms are bounded by $\rho(S + 1)$ and $\sigma(S + 1)$ respectively. The third term is bounded by a constant since $f$ and $g$ are close on $A \cap B$. Finally, the second term is bounded by $\rho(2S + 2 + R)$ since
$$
d(c_1, c_2) \leq d(a, c_1) + d(a, b) + d(b, c_2).
$$
This shows that $h$ is bornologous (since $S$ depends only on $R$), and properness of $h$ is easy to check.

$(\Leftarrow)$: Define a new metric $d'$ on $X$ as follows:
$$
d'(a,b) = \begin{cases} 
      d(a, b) & a,b \in A \setminus B \\
      d(a,b)  & a,b \in B \setminus A\\
      \mathsf{inf}\{ d(a, c) + d(c, b)\mid c \in A \cap B\} & a\in A, b\in B \\
   \end{cases}
$$
One checks that this is a metric. Consider the inclusions $i: A \rightarrow (X, d')$, $j: B \rightarrow (X, d')$. They are actually isometric embeddings, and hence coarse. The maps $i$ and $j$ agree on $A \cap B$, so by the universal property of the pushout, there must be a coarse map $h: X \rightarrow (X, d')$ which is close to the identity. Since maps which are close to bornologous maps are bornologous, we may assume that $h$ \emph{is} the identity. Suppose that $h$ is $\rho$-bornologous. Let $R > 0$, and let $x \in B(A, R) \cap B(B, R)$. Without loss of generality, suppose that $x \in A$, and that $d(x, b) \leq 2R$ for $b \in B$. Since $h$ is bornologous, we have (by definition of the metric $d'$) that $x$ must be at most $\rho(2R) + 1$ away from $A \cap B$. Thus we can set $S = \rho(2R) + 1$.
\end{proof}

Note that in the coarse category of metric spaces, there is at most one morphism from a bounded space $K$ to a metric space $X$ (since any two coarse maps $K \rightarrow X$ are close). Thus, by general category theoretic arguments, if $A \cap B$ is bounded then the diagram (\ref{pushout}) is a pushout if and only if $X$ (together with the inclusions) is the coproduct of $A$ and $B$. Furthermore, note that in condition (b) of Theorem~\ref{disconthm}, one may choose the sets $A$ and $B$ to have nonempty intersection and to be closed (simply take all points which are at most $R$ distance from $A$ and $B$ respectively, for a suitable $R$). This leads to the following corollary, which can also be verified directly using condition (b) in Theorem~\ref{disconthm}.

\begin{Corollary}\label{decomphigson}
A metric space $X$ satisfies $\mathsf{(C)}$ (or equivalently, has a connected Higson corona) if and only if in every $\omega$-excisive decomposition $X = A \cup B$ of $X$ with $A \cap B$ bounded, one of $A$ and $B$ is bounded. 
\end{Corollary}


The fact that proper metric spaces which do not satisfy $\mathsf{(C)}$ have disconnected Higson coronas can now be seen as a direct consequence of Proposition 1 in~\cite{HigsonMV}, while Lemma~\ref{slowcop} in this paper can be seen as a consequence of Proposition 2 in~\cite{HigsonMV} for the proper case (noting that every coarse map is slowly oscillating on a bounded subset). Theorem~\ref{discon1} of the present paper states in part that disconnectedness of the Higson corona $\nu X$ ensures the existence of a $\omega$-excisive decomposition $X = A \cup B$ of $X$ with $A$ and $B$ unbounded and $A \cap B$ bounded.

\section{Cohomological characterisation}\label{Seccohom}
In this section, we prove the following:

\begin{Theorem}\label{cohom}
A metric space $M$ has a connected Higson corona if and only if its first coarse cohomology group $HX^1(M)$ is trivial.
\end{Theorem}

We briefly recall the definition of coarse cohomology, following~\cite{RoeLectures}. For $M$ a metric space and $q$ a natural number, a subset $E \subseteq M^{q+1}$ is called \emph{controlled} if all the product projections $\pi_1, \ldots, \pi_{q+1}$ are close on $E$. The subset $E$ is called \emph{bounded} if every product projection is close to a constant map. Note that if $q = 0$, then every subset is controlled, while the bounded sets are precisely the bounded sets in the usual sense.

\begin{Definition}
Let $M$ be a metric space. A subset $D \subseteq M^{q+1}$ is called \emph{cocontrolled} if its intersection with every controlled set $E \subseteq M^{q+1}$ is bounded.
\end{Definition}

In the case of $q = 0$, the cocontrolled subsets are precisely the bounded ones. Given an abelian group $G$, the \emph{coarse complex of $M$ with coefficients in $G$}, denoted by $CX^\ast(M; G)$, is defined as the space of functions $\phi: M^{q+1} \rightarrow G$ with cocontrolled support. The complex can be equipped with coboundary maps $\delta: CX^{q+1}(M; G) \rightarrow CX^{q+2}(M; G)$ defined as follows:
$$
\delta \phi(x_0, \ldots, x_{q+1}) = \sum_{i=0}^{q+1} (-1)^i \phi(x_0, \ldots, \hat{x}_i, \ldots, x_{q+1}),
$$
where the `hat' denotes omission of a specific term. One checks that this defines a cochain complex, and the \emph{coarse cohomology} $HX^\ast(M; G)$ is defined to be the cohomology of this complex. When $G = \mathbb{Z}$, we denote the cohomology simply by $HX^\ast(M)$.

\begin{proof}[Proof of Theorem~\ref{cohom}]
($\Rightarrow$): Suppose $f: M^2 \rightarrow \mathbb{Z}$ represents a non-trivial cohomology class in $HX^1(M)$. In other words, $f$ satisfies $f(b, c) - f(a, c) + f(a, b) = 0$ (because $\delta f = 0$) but cannot be written as $f(a,b) = g(a) - g(b)$ for any function $g: M \rightarrow \mathbb{Z}$ with cocontrolled (equivalently, bounded) support. Define a relation $R$ on $M$ as follows: 
$$aRb \Leftrightarrow f(a,b) = 0.$$
It follows from the conditions on $f$ that $R$ is an equivalence relation. We claim that any equivalence class of $R$ has an unbounded complement. Suppose not; let $A$ be an equivalence class with bounded complement and pick $a \in A$. Define a function $g: M \rightarrow \mathbb{Z}$ as follows:
$$
g(x) = f(x,a).
$$
Note that $g$ has bounded support, since the complement of $A$ is assumed to be bounded, and that $\delta g = f$, a contradiction. Thus we conclude that we can divide $M$ into two unbounded sets $A$ and $B$, each a union of equivalence classes. Let $S > 0$. Then $a \in A$, $b \in B$ and $d(a,b) \leq S$ implies
$$
(a, b) \in \mathsf{supp}(f) \cap \{(x,y) \mid d(x,y) \leq S\}
$$
where the intersection is bounded because $f$ has cocontrolled support. It follows that $A$ and $B$ satisfy the conditions in (b) of Theorem~\ref{disconthm}, which gives the required result. 

($\Leftarrow$): Let $M = A \cup B$ with $A$ and $B$ satisfying the conditions in (b) of Theorem~\ref{disconthm}. We may suppose that $A$ and $B$ are disjoint (simply take $B = M \setminus A$ in the event this is not the case). Let $f: M^2 \rightarrow \mathbb{Z}$ be the map
\[
f(x,y) = \begin{cases}
1 & x \in A, y \in B \\
-1 & x \in B, y \in A \\
0 & \mathrm{otherwise} 
\end{cases}
\]
It is easy to check that $f$ represents a non-trivial cohomology class in $HX^1(M)$. Indeed, the conditions on $A$ and $B$ force $f$ to have cocontrolled support, while the unboundedness of $A$ and $B$ ensure that $f$ is non-trivial in cohomology.
\end{proof}

Theorem~\ref{cohom} shows that coarse coproducts are not preserved (i.e.~taken to direct sums) when taking first cohomology groups, since a space may be the coproduct of two unbounded spaces each having a connected Higson corona. For proper metric spaces, one part of the proof above (namely that having a disconnected Higson corona implies non-trivial first cohomology), follows from Corollary~\ref{decomphigson} in the previous section and the result in~\cite{HigsonMV} that there is a long exact Mayer-Vietoris sequence
$$ \xymatrix{
\ldots \ar[r] & HX^0(A) \oplus HX^0(B) \ar[r] & HX^0(A \cap B) \ar[r] & HX^1(M) \ar[r] & \ldots 
}
$$
for any $\omega$-excisive decomposition $M = A \cup B$. Indeed, if $A$ and $B$ are unbounded and $A \cap B$ is bounded, then $HX^0(A)$ and $HX^0(B)$ are trivial, while $HX^0(A \cap B)$ is isomorphic to $\mathbb{Z}$ (see Section 5.1 of~\cite{RoeLectures}).

\section{Geodesic spaces and finitely generated groups}\label{SecGroups}
We expect that connectedness of the Higson corona should be the same as being ``connected at infinity'' for certain spaces. In this section we show such a result for the case of geodesic spaces. Recall that a metric space $X$ is said to be \emph{geodesic} (see for example~\cite{NowakYu}) if for any two points $x, y\in X$ there is an isometric embedding $\gamma$ of the interval $[0, d(x,y)]$ into $X$ with $\gamma(0) = x$, $\gamma(d(x,y)) = y$. We refer to the image of $\gamma$ as the geodesic from $x$ to $y$. 

\begin{Theorem}\label{geod1}
The following are equivalent for a geodesic metric space $X$:
\begin{itemize}
\item[(a)] the Higson corona of $X$ is (topologically) disconnected;
\item[(b)] there exists a bounded set $X_0 \subseteq X$ such that for any bounded set $C$ containing $X_0$, $X \setminus C$ is topologically disconnected.
\end{itemize}
\end{Theorem}
\begin{proof}
(a) $\Rightarrow$ (b): Suppose $X = A \cup B$ with $A$ and $B$ satisfying the conditions in (b) in Theorem~\ref{disconthm}. Let $X_0$ be the bounded set such that $a \in A \setminus X_0$, $b \in B\setminus X_0 \Rightarrow d(a,b) \geq 1$. Then the result follows easily from the fact that $A$ and $B$ are unbounded.

(b) $\Rightarrow$ (a): Let $X_0$ be as in (b). By the assumption on $X_0$, every connected component of $X \setminus X_0$ must have an unbounded complement in $X \setminus X_0$. It follows that we can divide $X \setminus X_0$ into two sets $A$ and $B'$, each a union of connected components. By connectedness, the geodesic from a point $a \in A \setminus X_0$ to a point $b \in B' \setminus X_0$ must pass through $X_0$. Thus, for any $R > 0$, the distance between points $a \in A \setminus B(X_0, R)$ and $b \in B' \setminus B(X_0, R)$ is at least $R$, so $A$ and $B = B' \cup X_0$ satisfy the conditions in (b) in Theorem~\ref{disconthm}.
\end{proof}

\begin{Example}
The above theorem shows that $\mathbb{R}^n$ has a connected Higson corona for $n \neq 1$, and a disconnected Higson corona when $n = 1$. Note, however, that by Theorem 5 of~\cite{Keesling}, the Higson corona of $\mathbb{R}^n$ is neither locally connected nor arcwise connected for $n > 1$. 
\end{Example}

We now consider the special case of finitely generated groups, seen as metric spaces. Let $G$ be a finitely generated group generated by a finite set $S$ which is closed under taking inverses. We can construct two coarsely equivalent metric spaces associated to the pair $(G, S)$:
\begin{itemize}
\item the group $G$ equipped with the \emph{word length metric}, i.e.~where $d(g,h)$ is the length of the minimal representation of $gh^{-1}$ using elements of $S$, and
\item the \emph{Cayley graph} $\Gamma(G, S)$ of $G$, i.e.~the graph with vertices the elements of $G$ and an edge from $g$ to $gs$ for every $g \in G$, $s \in S$, viewed as a 1-complex and equipped with the path length metric.
\end{itemize}

It turns out that both of these metric spaces do not depend, up to coarse equivalence, on the choice of finite generating set $S$ (see for example~\cite{NowakYu}). In particular, the Higson corona of a group $G$ with the word length metric is invariant under choice of generating set, and is moreover homeomorphic to the Higson corona of any Cayley graph associated to $G$. The \emph{number of ends} of $G$ is defined to be the number of (topological) ends of its Cayley graph. Note that the number of ends of the Cayley graph also does not depend on the finite generating set $S$ (see for example Section 13 of~\cite{Goeghegan} for more on ends of groups).

\begin{Corollary}\label{fingengroup}
A finitely generated group $G$ with the word length metric has a connected Higson corona if and only if it has at most one end.
\end{Corollary}
\begin{proof}
By the above remarks, a group $G$ with the word length metric has a connected Higson corona if and only if its Cayley graph $\Gamma(G, S)$ does. Note that $\Gamma(G, S)$ is a geodesic space. Fix a vertex $g$ in $\Gamma(G, S)$ and consider the cover of $\Gamma(G, S)$ by compact sets
$$
\overline{B(g, 1)} \subseteq \overline{B(g, 2)} \subseteq \overline{B(g, 3)} \subseteq  \cdots
$$
Here we use the fact that $\Gamma(G, S)$ is locally finite, i.e.~every vertex is an endpoint of finitely many edges. An end of $\Gamma(G, S)$ is then a sequence
$$
U_1 \supseteq U_2 \supseteq U_3 \supseteq \cdots
$$
where for each $i$, $U_i$ is a connected component of $\Gamma(G, S) \setminus \overline{B(g, i)}$. If $\Gamma(G, S)$ has more than one end, then there must be an $n$ such that $\Gamma(G, S) \setminus \overline{B(g, n)}$ has two unbounded connected components, in which case $\Gamma(G, S)$ has a disconnected Higson corona by Theorem~\ref{geod1}. Conversely, if $\Gamma(G, S)$ has a disconnected Higson corona, then by Theorem~\ref{geod1} there is some $n$ such that for any bounded set $K$ containing $\overline{B(g, n)}$, $\Gamma(G, S) \setminus K$ has at least two connected components. Notice that since $\Gamma(G, S)$ is locally finite, $\Gamma(G, S) \setminus \overline{B(g, n)}$ has finitely many connected components. It follows that two of its connected components must be unbounded, so that $\Gamma(G, S)$ has more than one end as required.
\end{proof}

Note that Corollary~\ref{fingengroup} also follows from Theorem 13.5.5 in~\cite{Goeghegan} and the fact that the usual group cohomology of a finitely generated group $G$ (with coefficients in the group ring $\mathbb{Z}G$) coincides with the coarse cohomology of the group as a metric space (see Example 5.21 of~\cite{RoeLectures}). 

\begin{Remark}
A complete characterisation of finitely generated groups with more than one end is already known. A finitely generated group can have either 0, 1, 2 or infinitely many ends; a finitely generated group $G$ has two ends if and only if $G$ has an infinite cyclic subgroup of finite index. The characterisation for infinitely many ends is given by a theorem of Stallings~\cite{Stallings68, Stallings71}. For more details and proofs of these facts, we refer the reader to~\cite{Goeghegan}. 
\end{Remark}

\section{Abstract coarse spaces}\label{SecAbstract}
In this section we consider the more general setting of coarse spaces. Most of the results of the previous sections generalise to this context, so long as one works with coarse spaces which are ``connected'' in the sense of~\cite{RoeLectures}, i.e.~in which finite sets are bounded. 

Recall from~\cite{RoeLectures} that a \emph{coarse space} is a pair $(X, \mathcal{X})$ where $X$ is a set and $\mathcal{X}$ is a family of binary relations on $X$ which contains the diagonal $\Delta$ and which is closed under taking subrelations, inverses, products (i.e.~composition of relations) and finite unions. A map $f$ between (the underlying sets of) coarse spaces $(A, \mathcal{A})$ and $(B, \mathcal{B})$ is called \emph{bornologous} if $(f \times f)(R) \in \mathcal{B}$ for every $R \in \mathcal{A}$. Given a coarse space $(A, \mathcal{A})$, a subset $B$ of $A$ is called \emph{bounded} if it is contained in $\{a \in A \mid aRx\}$ for some $R \in \mathcal{A}$ and $x \in A$. The notion of proper map can thus be defined for coarse spaces. Two maps $f,g: (A, \mathcal{A}) \rightarrow (B, \mathcal{B})$ are said to be \emph{close} if $\{(f(a), g(a)) \mid a \in A\}\in \mathcal{B}$. 

Every metric $d$ on a set $A$ induces a \emph{bounded coarse structure}, consisting of all those relations $R$ for which the set $\{d(a,b)\mid aRb\}$ is bounded. A map $f: A \rightarrow B$ between metric spaces is bornologous with respect to the metrics if and only if it is bornologous with respect to the respective bounded coarse structures. Thus $\mathbf{Met}_\mathbf{Born}$ can be viewed as a full subcategory of the category of coarse spaces and bornologous maps. A coarse space $(A, \mathcal{A})$ is called \emph{connected}~\cite{RoeLectures} if every finite subset of $A \times A$ is in $\mathcal{A}$. Bounded coarse structures associated to metrics (which are not allowed to take the value $\infty$) are always connected. We say that a coarse structure on a set $A$ is \emph{metrizable} if it is the bounded coarse structure associated to a metric on $A$. We recall the following result.

\begin{Theorem}[\cite{RoeLectures}]\label{countablygen}
A connected coarse structure $\mathcal{A}$ on $A$ is metrizable if and only if it is countably generated.
\end{Theorem}

Let $\mathcal{A}_0$ be a collection of relations on a set $A$. Then it is easy to show that there exists a smallest coarse structure on $A$ containing $\mathcal{A}_0$, which we denote by $\overline{\mathcal{A}_0}$. By ``countably generated'' in Theorem~\ref{countablygen}, we mean there is a countable set $\mathcal{A}_0$ of relations such that $\overline{\mathcal{A}_0} = \mathcal{A}$.

\begin{Lemma} \label{closure}
Let $\mathcal{A}$ be a set of relations on a set $A$ and let $f: A \rightarrow B$ be a map. Then we have
$$(f \times f)(\overline{\mathcal{A}}) \subseteq \overline{(f \times f)(\mathcal{A})}.$$
\end{Lemma}
\begin{proof}
This follows from the fact that for two relations $R$ and $S$ on $A$, 
$$
(f \times f)(R\circ  S) \subseteq (f \times f)(R) \circ (f \times f)(S).
$$
\end{proof}

\begin{Proposition}\label{coarsecop}
The category of connected coarse spaces and bornologous maps admits arbitrary coproducts. Moreover, the countable (or finite) coproduct of metrizable coarse spaces is metrizable.
\end{Proposition}
\begin{proof}
Suppose $(A_\alpha, \mathcal{A}_\alpha)_{\alpha \in I}$ is a family of connected coarse spaces. For each $A_\alpha$, pick a base point $a_\alpha \in \iota_\alpha(A_\alpha)$. Define a coarse structure $\sum_\alpha \mathcal{A}_\alpha$ on the disjoint union $\sum_\alpha A_\alpha$ of the $A_\alpha$ as follows:
$$
\sum_\alpha \mathcal{A}_\alpha = \overline{\bigcup_\alpha \mathcal{A}_\alpha \cup \bigcup_{\alpha \in I} \bigcup_{\alpha' \in I} \{(a_\alpha, a_{\alpha'})\}},
$$
Note that since each $(A_\alpha, \mathcal{A}_\alpha)$ is connected, $(\sum_\alpha A_\alpha, \sum_\alpha \mathcal{A}_\alpha)$ is also connected. Moreover, it is clearly the smallest connected coarse structure containing all the $\mathcal{A}_\alpha$, so using Lemma~\ref{closure} it is easy to check that this gives the required coproduct. 

Suppose now that $I = \mathbb{N}$ and that each $(A_i, \mathcal{A}_i)$ is metrizable, with each $\mathcal{A}_i$ generated by relations $\{\mathcal{A}_{i, 1}, \mathcal{A}_{i, 2}, \ldots \}$. Note that the set
$$
\bigcup_{i \in \mathbb{N}} \bigcup_{j \in \mathbb{N}} \{(a_i, a_{j})\},
$$
is countable, as is the set
$$
\{ \bigcup_{i = 1}^{k-1} \mathcal{A}_{i, k-i} \mid k \in \{2,3 \ldots\}  \}
$$
and together they generate $\sum_\alpha \mathcal{A}_\alpha$. Thus by Theorem~\ref{countablygen}, $\sum_\alpha \mathcal{A}_\alpha$ is metrizable.
\end{proof}
It is easy to check that the construction of binary coproducts given in the proposition above gives the binary coproduct in the subcategory of coarse maps as well as in the \emph{coarse category of connected coarse spaces}, i.e.~the category whose objects are connected coarse spaces and whose morphisms are equivalence classes of coarse maps under the closeness relation. 

Throughout the rest of this section, by a \emph{coarse coproduct} of two connected coarse spaces, we mean the coproduct of the two spaces in the category of connected coarse spaces and bornologous maps. By the proof of Proposition~\ref{coarsecop}, the coarse coproduct of $(X, \mathcal{X})$ and $(Y, \mathcal{Y})$ is given by 
$$(X + Y, \overline{\mathcal{X} \cup \mathcal{Y} \cup \{(x_0, y_0)\} }),$$
where $x_0 \in X$ and $y_0 \in Y$ are arbitrary chosen base points.

Proposition~\ref{coarsecop} gives an alternative proof that the category of metric spaces and bornologous maps admits finite and countable coproducts. The explicit construction of the metric $d$ in Proposition~\ref{countablecop} can be derived from the proof of Theorem~\ref{countablygen} (see~\cite{RoeLectures}) and the second half of the proof of Proposition~\ref{coarsecop} above. In the case of binary coproducts (see Proposition~\ref{coprod1}), the description of the metric follows naturally from the following lemma. In order to state the lemma, we introduce the following notation. If $X$ and $Y$ are sets and $R$ is a relation on $X$, we denote by $R_r$ the smallest reflexive relation on the disjoint union $X + Y$ whose restriction to $X$ is $R$ (it is nothing but the union of $R$ with the diagonal relation on $X + Y$).

\begin{Lemma} \label{techcopr} 
Let $(X + Y, \mathcal{X} + \mathcal{Y})$ be the coarse coproduct of $(X, \mathcal{X})$ and $(Y, \mathcal{Y})$, and let $x_0 \in X$, $y_0 \in Y$. Then every relation in $\mathcal{X} + \mathcal{Y}$ is contained in a relation of the form
$$
(R_r \circ U \circ S_r) \cup (S_r \circ U \circ R_r),
$$
where $R$ and $S$ are in $\mathcal{X}$ and $\mathcal{Y}$ respectively, and $U = \{(x_0, y_0), (y_0, x_0)\} \cup \Delta$. 
\end{Lemma}
\begin{proof}
Let $\mathcal{Z}_0$ be the set of all relations of the form
$$
(R_r \circ U \circ S_r) \cup (S_r \circ U \circ R_r),
$$
for $R \in \mathcal{X}$ and $S \in \mathcal{Y}$, and let $\mathcal{Z}_1$ be its closure under taking subrelations. Since $\mathcal{X}, \mathcal{Y} \subseteq \mathcal{Z}_1 \subseteq \mathcal{X} + \mathcal{Y}$, it is enough to show that $\mathcal{Z}_1$ is a coarse structure by minimality of $\mathcal{X} + \mathcal{Y}$. Some straightforward computations show that $\mathcal{Z}_0$ is closed under composition of relations, and the result follows. 
\end{proof}

Since condition $\mathsf{(C)}$ in Section~\ref{SecHigson} was stated in the language of coarse coproducts, coarse maps and closeness of maps, the definition extends to the case of general coarse spaces. Theorem~\ref{disconthm} generalizes partially to this setting, as shown in the following theorem. The proof of the theorem is a straightforward adaptation of the proof of Theorem~\ref{disconthm} now that we have Lemma~\ref{techcopr}, so we omit it.

\begin{Theorem}\label{absdisconthm}
For a connected coarse space $(X, \mathcal{X})$, the following are equivalent:
\begin{itemize}
\item[(a)] $X$ doesn't satisfy $\mathsf{(C)}$;
\item[(b)] there are two unbounded subsets $A$ and $B$ of $X$ such that
\begin{itemize}
\item $X = A \cup B$, and
\item for any $R \in \mathcal{X}$, there is a bounded set $C_R$ such that if $xRx'$ for $x, x' \in X \setminus C_R$, then $\{x, x'\}$ intersects at most one of $A$ and $B$,
\end{itemize}
\item[(c)] is (bijectively) coarsely equivalent to a coarse coproduct $(Y + Z, \mathcal{Y} + \mathcal{Z})$ where neither $Y$ nor $Z$ is bounded.
\end{itemize}
\end{Theorem}

\begin{Example}
Let $X$ be any infinite set and let $\mathcal{X}$ be the smallest connected coarse structure on $X$. In other words, $\mathcal{X}$ consists of all finite subsets of $X \times X$ together with all relations of the form $R \cup \Delta$, where $R$ is a finite subset of $X \times X$. It is easy to show that $(X, \mathcal{X})$ does not satisfy $\mathsf{(C)}$ (using for example (b) from Theorem~\ref{absdisconthm} and noting that finite sets are always bounded in a connected coarse space), and thus has a disconnected Higson corona.
\end{Example}

The above example shows that we cannot hope for a version of (d) from Theorem~\ref{disconthm} to appear in the above theorem. Indeed, take $X$ to be any uncountable set and $\mathcal{X}$ the smallest connected coarse structure on $X$. Then there are no proper maps from $(X, \mathcal{X})$ to $\mathbb{Z}$.

A map $f: (X, \mathcal{X}) \rightarrow Y$ from a coarse space to a metric space $Y$ is called \emph{slowly oscillating} if for every $R \in \mathcal{X}$ and $\varepsilon > 0$, there exists a bounded set $K \subseteq X$ such that for any $a,b \in X$, $a, b \notin K$,
$$
aRb \Rightarrow d(f(a), f(b)) \leq \varepsilon.
$$
This, together with a similar definition for tending to $0$ at infinity, allows one to define the  $C^\ast$-algebra $B_h(X)/B_0(X)$ for any coarse space, and consequently also the Higson corona~\cite{RoeLectures}. We have the following generalization of the main theorem.

\begin{Theorem}\label{disconcoarse1}
The following are equivalent for a connected coarse space $(X, \mathcal{X})$.
\begin{itemize}
\item[(a)] $(X, \mathcal{X})$ satisfies $\mathsf{(C)}$;
\item[(b)] the Higson corona of $X$ is (topologically) connected;
\item[(c)] $B_h(X)/B_0(X)$ does not contain a non-trivial idempotent element.
\end{itemize}
\end{Theorem}
\begin{proof}
This is an easy adaptation of the metric case now that we have Theorem~\ref{absdisconthm} (and in particular, part (b)). Note that for connected coarse spaces, the finite union of bounded sets is bounded.
\end{proof}

\section*{Acknowledgements}
This paper is based on research done while the author was a PhD student under the supervision of Jerzy Dydak at the University of Tennessee.

\end{document}